\documentclass[english]{smfart}
\usepackage{sabbah_irreg-complex3}
\begin{document}
\frontmatter
\title{A remark on the irregularity complex}

\author[C.~Sabbah]{Claude Sabbah}
\address[C.~Sabbah]{CMLS, École polytechnique, CNRS, Université Paris-Saclay\\
F--91128 Palaiseau cedex\\
France}
\email{Claude.Sabbah@polytechnique.edu}
\urladdr{http://www.math.polytechnique.fr/perso/sabbah}
\thanks{This research was supported by the grant ANR-13-IS01-0001-01 of the Agence nationale de la recherche.}


\begin{abstract}
We prove that, for a good meromorphic flat bundle with poles along a divisor with normal crossings, the restriction of the irregularity complex to each natural stratum of this divisor only depends on the formal flat bundle along this stratum. This answers a question raised by J.-B.\,Teyssier. The present version takes into account the published erratum of \cite{Bibi16b}.
\end{abstract}

\maketitle
\tableofcontents
\mainmatter

\section{Statement of the results}
Let $X$ be a complex manifold of dimension $n$ and let $D=\bigcup_{i\in J}D_i$ be a divisor with normal crossings. We assume that each irreducible component $D_i$ of $D$ is smooth. For any subset $I\subset J$, we set $D_I=\bigcap_{i\in I}D_i$ and $D_I^\circ=D_I\moins\bigcup_{j\notin I}D_j$. We denote the codimension of $D_I^\circ$ by $\ell$, that we regard as a locally constant function on $D_I^\circ$ (which can have many connected components), and by $\iota_I:D_I^\circ\hto D$ the inclusion. Let~$\cM$ be a holonomic $\cD_X$-module such that
\begin{enumerate}
\item
$\cM=\cM(*D)$,
\item
$\cM_{X\moins D}$ is locally $\cO_X$-free of finite rank.
\end{enumerate}
We then say that $\cM$ is a meromorphic flat bundle with poles along $D$. \emph{In this note, we assume that $\cM$ has a good formal structure along $D$} (we simply say that~$\cM$ is a \emph{good $D$-meromorphic flat bundle}, or a \emph{good meromorphic flat bundle on $(X,D)$}). This notion, together with the Riemann-Hilbert correspondence, will be recalled in Section~\ref{sec:bonnestructureRH}. Recall also that, given any meromorphic flat bundle on $(X',D')$ (where $D'$ is an arbitrary reduced divisor in $X'$), there exists, locally on $X'$, a projective modification $X\to X'$ such that the pullback of $D'$ by this modification is a divisor with simple normal crossings $D$ and the pullback meromorphic flat bundle is a $D$-meromorphic flat bundle having a good formal structure along $D$ (\cf\cite{Kedlaya09,Kedlaya10}, and \cite{Mochizuki07b, Mochizuki08} in the algebraic case; see also \cite{Bibi97} for special cases when $\dim X=2$).

For every $I\subset J$, we consider the sheaf $\cO_{\wh{X|D_I^\circ}}$ on $D_I^\circ$, also denoted by $\cO_{\wh{D_I^\circ}}$, defined as the formalization of $\cO_X$ along $D_I^\circ$. We also regard it as a sheaf on $X$ by extending it by zero. We then set $\cD_{\wh{D_I^\circ}}=\cO_{\wh{D_I^\circ}}\otimes_{\cO_X}\cD_X$, and $\cM_{\wh{D_I^\circ}}:=\cD_{\wh{D_I^\circ}}\otimes_{\cD_X}\cM$.

For any holonomic $\cD_X$-module $\cN$, the irregularity complexes $\Irr_D\cN$ and $\Irr^*_D\cN$, as defined by Mebkhout \cite{Mebkhout90}, are constructible complexes supported on $D$, and only depend on $\cN(*D)$. For a \emph{good} $D$-meromorphic flat bundle $\cM$ as above, the cohomology of $\Irr_D\cM$ and $\Irr^*_D\cM$ is \hbox{locally} constant along each stratum $D_I^\circ$: this follows from \cite[Th.\,12.2.7]{Teyssier13} if $\#I=1$ and from Corollary \ref{cor:irrrestr} together with the case $\#I=1$ otherwise. On the other hand, Mebkhout has shown that the complexes $\Irr_D\cM[\dim X]$, $\Irr^*_D\cM[\dim X]$ are a perverse sheaves (\cf \loccit).

Our aim in this note is to compare the irregularity complexes of $\cM$ restricted to~$D_I^\circ$ and those of the formalized module $\cM_{\wh{D_I^\circ}}$. However, the irregularity complexes of $\cM_{\wh{D_I^\circ}}$ are not defined by the procedure of \cite{Mebkhout90}. To give a meaning to the question, we start by proving in Section \ref{subsec:proofprop1} the following proposition.

\begin{proposition}\label{prop:MMhat}
For every $I\subset J$, there exists a unique good $D$-meromorphic flat bundle~$\cM_I^\circ$ in the neighbourhood of $D_I^\circ$ which satisfies the following two properties.
\begin{enumerate}
\item\label{prop:MMhat1}
$\cD_{\wh{D_I^\circ}}\otimes_{\cD_X}\cM_I^\circ\simeq\cM_{\wh{D_I^\circ}}$.
\item\label{prop:MMhat2}
At each point of $D_I^\circ$, the local formal decomposition of $\cM_I^\circ$ (after a local ramification around $D$) into elementary formal $D$-meromorphic flat bundles already holds without taking formalization.
\end{enumerate}
\end{proposition}

The main result of this note can now be stated.

\begin{theoreme}\label{th:main}
For every $I\subset J$, we have
\[
\iota_I^{-1}\Irr_D\cM\simeq\iota_I^{-1}\Irr_D(\cM_I^\circ),\quad\text{and}\quad\iota_I^{-1}\Irr^*_D\cM\simeq\iota_I^{-1}\Irr^*_D(\cM_I^\circ).
\]
\end{theoreme}

In other words, the complexes $\iota_I^{-1}\Irr_D\cM,\iota_I^{-1}\Irr^*_D\cM$ only depend (up to isomorphism) on the formalization $\cM_{\wh{D_I^\circ}}$ of $\cM$ along $D_I^\circ$.

\subsubsection*{Acknowledgements}
The statement of Theorem \ref{th:main} has been suggested, in a numerical variant, by Jean-Baptiste Teyssier, against my first expectation. He was motivated by a nice application to moduli of Stokes torsors  obtained in \cite{Teyssier16}. I thank him for having led me to a better understanding of the irregularity complex, and for suggesting a simpler proof of Proposition \ref{prop:MMhat}. I thank the referee for interesting comments.

\section{Good formal structure and the Riemann-Hilbert correspondence}\label{sec:bonnestructureRH}
\subsection{Notation}\label{subsec:nota}
We keep the notation of the introduction. If $Z$ is any locally closed analytic subspace of the complex analytic manifold $X$, we denote by $\cO_{\wh Z}$, the formal completion of $\cO_X$ with respect to the ideal sheaf $\cI_Z$. We regard $\cO_{\wh Z}$ as a sheaf on~$Z$.

Given $x_o\in D$, there exists a unique $I\subset J$ such that $x_o\in D_I^\circ$, and we will be mostly interested in the case where $Z$ is the point $x_o\in D$ and the case where $Z$ is equal to $D_I^\circ$. We will denote by $\cO_{\wh Z}(*D)$ the sheaf $\cO_{X|Z}(*D)\otimes_{\cO_{X|Z}}\cO_{\wh Z}$, where as usual $\cO_{X|Z}$ (\resp $\cO_{X|Z}(*D)$) denotes the sheaf-theoretic restriction to $Z$ of the sheaf~$\cO_X$ of holomorphic functions on $X$ (\resp the sheaf $\cO_X(*D)$ of meromorphic functions on~$X$ with poles at most on $D$).

If $\varphi$ (\resp $\wh\varphi$) is a section of $\cO_X(*D)$ (\resp of $\cO_{\wh Z}(*D)$), we denote by $\cE^\varphi$ (\resp $\cE^{\wh\varphi}$) the module with connection $(\cO_X(*D),\rd+\rd\varphi)$ (\resp $(\cO_{\wh Z}(*D),\rd+\wh\varphi)$). It only depends on the class, also denoted by $\varphi$ (\resp $\wh\varphi$), of $\varphi$ (\resp $\wh\varphi$) modulo $\cO_X$ (\resp $\cO_{\wh Z}$).

\subsection{Good formal structure}
We say that the $D$-meromorphic flat bundle $\cM$ has a \emph{good formal structure} if, for any $x_o\in D$, there exists a local ramification $\rho_{\bmd_I}$ of multi-degree $\bmd_I$ around the branches $(D_i)_{i\in I}$ passing through $x_o$ (hence inducing an isomorphism above $D_I^\circ$ in the neighbourhood of $x_o$) such that the pullback of the formal flat bundle $\cM_{\wh{x_o}}:=\cO_{\wh x_o}\otimes_{\cO_{X,x_o}}\cM_{x_o}$ by this ramification decomposes as the direct sum of formal elementary $D$-meromorphic connections $\cE^{\wh\varphi}\otimes\wh\cR_{\wh\varphi}$, as defined below.

We denote by $\nb(x_o)$ a small open neighbourhood of $x_o$ in $X$ above which the ramification is defined, and we denote by $x'_o$ the pre-image of $x_o$, so the ramification is a finite morphism $\rho_{\bmd_I}:\nb(x'_o)\to\nb(x_o)$. It induces a one-to-one map above $D_I^\circ\cap\nb(x_o)$. We also set $D'=\rho_{\bmd_I}^{-1}(D\cap\nb(x_o))$, so that $D'_I$ maps isomorphically to $D_I\cap\nb(x_o)=D_I^\circ\cap\nb(x_o)$.

In the above decomposition, $\wh\varphi$ varies in a \emph{good} finite subset $\wh\Phi_{x_o}\subset\cO_{\wh{x'_o}}(*D')/\cO_{\wh{x'_o}}$ and~$\wh\cR_{\wh\varphi}$ is a free $\cO_{\wh{D'_I}}(*D')$-module with an integrable connection having a regular singularity along $D'$. In other words, we do not distinguish between $\wh\varphi$ and $\wh\psi$ in $\cO_{\wh{x'_o}}(*D')$ if their difference has no poles along $D'$. Goodness means here that for any pair \hbox{$\wh\varphi\neq\wh\psi\in\wh\Phi_{x_o}\cup\{0\}$}, the difference $\wh\varphi-\wh\psi$ can be written as $x^{-\bmm}\wh\eta(x)$, with $\bmm\in\NN^{\#I}$ and $\wh\eta\in\cO_{\wh{x'_o}}$ satisfying $\wh\eta(0)\neq0$ (\cf \cite[\S I.2.1]{Bibi97}.\footnote{ Note that, here, the goodness condition is assumed for $\wh\Phi\cup\{0\}$ and not only for~$\wh\Phi$, because of \cite[Cor.\,12.7]{Bibi10}. This is unfortunately not made precise in \cite[Th.\,12.16]{Bibi10} and should be corrected.} By \cite[Prop.\,4.4.1\,\&\,Def.\,5.1.1]{Kedlaya10} (\cf also \cite[\S I.2.4]{Bibi97} and \cite[Prop.\,2.19]{Mochizuki10b}), the~$\wh\varphi$'s are convergent, \ie the set~$\wh\Phi_{x_o}$ is the formalization at $x_o$ of a finite subset $\Phi_{x_o}\subset\nobreak\Gamma(\nb(x'_o),\cO_{\nb(x'_o)}(*D')/\cO_{\nb(x'_o)})$, and the decomposition extends in a neighbourhood of $x'_o$, that is, it holds for the pullback by $\rho_{\bmd_I}$ of $\cM_{\wh{D_I^\circ},x_o}$ and induces the original one after taking formalization at $x'_o$.\footnote{I thank J.-B.\,Teyssier for pointing this out to me. In \cite{Mochizuki08,Mochizuki10b} (\cf also \cite[\S11.3]{Bibi10}), this is shown to hold only if one assumes the good formal structure at all points of $D_I^\circ\cap\nb(x_o)$.}

\subsection{Stratified $\ccI$-covering}
The set $\bigsqcup_{x_o\in D_I^\circ}(\Phi_{x_o}\cup\{0\})$ has a natural structure of a finite non-ramified covering of $D_I^\circ$ (in particular, it is a Hausdorff topological space), that we denote by $\Sigma_I^\circ\to D_I^\circ$. Locally, it is described as follows. Given a germ $\varphi_{x'_o}\in \Phi_{x_o}\cup\{0\}$, it extends locally as a section of $\cO_{\nb(x'_o)}(*D')/\cO_{\nb(x'_o)}$ and thus defines a germ in $\Phi_{y_o}\cup\{0\}$ for any $y_o\in D_I^\circ\cap\nb(x_o)$. This defines the local branch of $\Sigma_I^\circ$ passing through $\varphi_{x'_o}$. (This construction is nothing but that of the sheaf space, or étalé space, of a sheaf.)

By a similar procedure, the set $\Sigma(\cM):=\bigsqcup_I\Sigma_I$ can be endowed with a natural topology as a sheaf space, but the topology can be non-Hausdorff: this occurs if some difference $\varphi_{x'_o}-\psi_{x'_o}$ does not have poles along \emph{all} the components of $D'$ passing through $x'_o$.

In order to state the Riemann-Hilbert correspondence, we will lift these objects to the \emph{real oriented blowing-up $\varpi:\wt X:=\wt X(D_{i\in J})\to X$} along the components $D_i$ of~$D$ in $X$. We set $\partial\wt X:=\varpi^{-1}(D)$ and $\partial\wt X_I^\circ:=\varpi^{-1}(D_I^\circ)$. The fibre of $\varpi$ over a point in~$D_I^\circ$ is diffeomorphic to~$(S^1)^\ell$, making $\partial\wt X_I^\circ$ a $(S^1)^\ell$-bundle on $D_I^\circ$.  We consider the sheaf $\ccI$ on $\partial\wt X$ as constructed in \cite[\S9.3]{Bibi10}. 

By considering the fiber product
\[
\xymatrix{
\wt\Sigma_I^\circ\ar[r]\ar[d]&\Sigma_I^\circ\ar[d]\\
\partial\wt X_I^\circ\ar[r]^-\varpi&D_I^\circ
}
\]
we obtain a finite covering $\wt\Sigma_I^\circ$ of $\partial\wt X_I^\circ$ which is naturally contained in the étalé space~$\ccI^{\textup{ét}}$ of $\ccI$. By a similar procedure, we get a \emph{good stratified $\ccI$-covering} $\bigsqcup_I\wt\Sigma_I^\circ=:\wt\Sigma(\cM)\to\partial\wt X$ of $\partial\wt X$, in the sense of \cite[Rem.\,11.12]{Bibi10}. As before, $\wt\Sigma(\cM)$ can be non-Hausdorff.

\subsection{The Riemann-Hilbert correspondence (local theory)}\label{subsec:RHlocal}
Let us fix a \emph{good stratified $\ccI$\nobreakdash-covering} $\wt\Sigma$. Let $x_o\in D_I^\circ$. The local Riemann-Hilbert correspondence (\cite{Mochizuki08,Mochizuki10b}, \cite{Bibi10}) is an equivalence between the category of germs at $x_o$ of good $D$-meromorphic flat bundles~$\cM_{x_o}$ with stratified $\ccI$-covering $\wt\Sigma(\cM)$ contained in $\wt\Sigma$, and that of germs at $\varpi^{-1}(x_o)$ of good Stokes-filtered local systems $(\cL_I^\circ,\cL_{I,\bbullet}^\circ)$ on $\partial\wt X_I^\circ$ (\cf\eg \cite[\S9.5]{Bibi10}) with $\ccI$-covering contained in $\wt\Sigma_I^\circ$ (\cf\cite[Th.\,4.11]{Mochizuki10b} and \cite[Th.\,12.16]{Bibi10}).

More precisely, we have a commutative diagram of functors
\begin{equation}\label{eq:RHlocal}
\begin{array}{c}
\xymatrix{
\cM_{x_o}\ar@{|->}[d]\ar@{|->}[r]^-\sim&(\cL_I^\circ,\cL_{I,\bbullet}^\circ)_{\varpi_{-1}(x_o)}\ar@{|->}[d]^\gr\\
\cM_{\wh{D_I^\circ},x_o}\ar@{|->}[r]^-\sim&(\gr\cL_I^\circ,\gr\cL_{I,\bbullet}^\circ)_{\varpi_{-1}(x_o)}
}
\end{array}
\end{equation}
similar to that of \cite[p.\,58]{Malgrange91}, where $\gr$ means grading with respect to the Stokes filtration and the horizontal functors are equivalences of categories. Recall that grading a Stokes-filtered local system is well-defined only when one restricts to $\wt\Sigma_I^\circ$, which is Hausdorff (\cf\cite[Chap.\,1]{Bibi10}). In order to give a meaning to grading in general, one needs to control the extension from $D_I^\circ$ to a small neighbourhood $\nb(D_I^\circ)$. Locally, this is provided by the following equivalence.

\begin{proposition}[\cf {\cite[Lem.\,3.17]{Mochizuki10b}}]\label{prop:SigmaI}
The restriction functor to $\partial\wt X_I^\circ$ induces an equivalence between the category of germs at $\varpi^{-1}(x_o)$ of Stokes-filtered local systems $(\cL,\cL_\bbullet)$ on $\partial\wt X$ with associated stratified $\ccI$-covering contained in $\wt\Sigma$ and the category of germs at $\varpi^{-1}(x_o)$ of Stokes-filtered local systems $(\cL_I^\circ,\cL_{I,\bbullet}^\circ)$ on $\partial\wt X_I^\circ$ with associated $\ccI$-covering contained in $\wt\Sigma_I^\circ$.
\end{proposition}

\subsection{The Riemann-Hilbert correspondence (global theory)}\label{subsec:RHglobal}
We now consider the previous correspondence all along $D_I^\circ$. We consider a covering $\cU$ of $D_I^\circ$ by open subsets $U_\alpha$ which are the intersection of $D_I^\circ$ with a local chart on $X$. Any germ~$\cM$ of $D$-meromorphic flat bundle along $D_I^\circ$ gives rise to gluing data $((\cM_\alpha),(\sigma_{\alpha\beta}))$, where $\cM_\alpha=\cM_{|U_\alpha}$, $\sigma_{\alpha\beta}:\cM_{\alpha|U_\alpha\cap U_\beta}\to\cM_{\beta|U_\alpha\cap U_\beta}$ is an isomorphism, and the family $(\sigma_{\alpha\beta})$ satisfies the cocycle property. Any germ $\cM$ of good $D$-meromorphic flat bundle along~$D_I^\circ$ admits a covering $\cU$ such that one can apply the local Riemann-Hilbert correspondence of Section \ref{subsec:RHlocal} to its restriction $\cM_\alpha$ to every $U_\alpha$. Given such a covering~$\cU$, we can consider the category of such good gluing data $\big((\cM_\alpha),(\sigma_{\alpha\beta})\big)$. The local Riemann-Hilbert correspondence gives rise to a commutative diagram of functors between gluing data
\begin{equation}\label{eq:RHglobal}
\begin{array}{c}
\xymatrix{
\big((\cM_\alpha),(\sigma_{\alpha\beta})\big)\ar@{|->}[d]\ar@{|->}[r]&\big((\cL_I^\circ,\cL_{I,\bbullet}^\circ)_{\alpha},(\eta_{\alpha\beta})\big)\ar@{|->}[d]^\gr\\
\big((\cM_{\alpha,\wh D_I^\circ}),(\wh\sigma_{\alpha\beta})\big)\ar@{|->}[r]&\big((\gr\cL_I^\circ,\gr\cL_{I,\bbullet}^\circ)_{\alpha},(\gr\eta_{\alpha\beta})\big)
}
\end{array}
\end{equation}
and the horizontal functors remain equivalences, due to the full faithfulness of the horizontal functors in \eqref{eq:RHlocal}.

Arguing similarly with the equivalence of Proposition \ref{prop:SigmaI}, we obtain the Riemann-Hilbert correspondence.

\begin{theoreme}\label{th:RHDIcirc}
The category $\Mod_\hol\bigl((X,D_I^\circ),D,\wt\Sigma\bigr)$ of germs along $D_I^\circ$ of good $D$-mero\-morphic flat bundles with stratified $\ccI$-covering contained in $\wt\Sigma$ is equivalent to that of germs along $\partial\wt X_I^\circ$ of Stokes-filtered local systems $(\cL,\cL_\bbullet)$ on $\partial\wt X$ with associated stratified $\ccI$-covering contained in~$\wt\Sigma$ and, by restriction, to that of Stokes-filtered local systems $(\cL_I^\circ,\cL_{I,\bbullet}^\circ)$ on $\partial\wt X_I^\circ$ with associated $\ccI$-covering contained in~$\wt\Sigma_I^\circ$.\qed
\end{theoreme}

\subsection{Proof of Proposition \ref{prop:MMhat}}\label{subsec:proofprop1}

By Theorem \ref{th:RHDIcirc}, there exists a germ $\cM_I^\circ$ along~$D_I^\circ$ of good $D$-meromorphic flat bundle whose associated Stokes-filtered local system is $(\gr\cL_I^\circ,\gr\cL_{I,\bbullet}^\circ)$, and it is unique up to isomorphism with respect to this property. A~covering $\cU$ adapted to $\cM$ is also adapted to $\cM_I^\circ$, and the diagram \eqref{eq:RHglobal} shows that the gluing data of $\cM_{\wh D_I^\circ}$ and of $\cM_{I,\wh D_I^\circ}^\circ$ are isomorphic, since they correspond to the same Stokes-filtered gluing data $\big((\gr\cL_I^\circ,\gr\cL_{I,\bbullet}^\circ)_{\alpha},(\gr\eta_{\alpha\beta})\big)$. The uniqueness of~$\cM_I^\circ$ is proved similarly.\qed

\begin{remarque}
The construction of $\cM_I^\circ$ is functorial with respect to $\cM_{|D_I^\circ}$.
\end{remarque}

\subsection{An equivalence of categories}\label{sec:equivcat}
Let $\catA$ be a category and let $G$ be a group. The category $\GcatA$ is the category whose objects are $G$-objects of $\catA$, that is, pairs $(M,\rho)$ where $M$ is an object of $\catA$ and $\rho$ is a morphism $G\to\Aut(M)$, and for which $\Hom_{\GcatA}((M,\rho_M),(N,\rho_N))\subset\Hom_{\catA}(M,N)$ is the subset consisting of morphisms $\varphi:M\to N$ such that, for every $g\in G$, $\varphi\circ\rho_M(g)=\rho_N(g)$.

Let $\wt\Sigma\to\partial\wt X$ be a good stratified $\ccI$-covering and let $\Mod_\hol(X,D,\wt\Sigma)$ denote the full subcategory of that of holonomic $\cD_X$-modules whose objects consist of good meromorphic flat bundles on $(X,D)$ with associated stratified $\ccI$-covering contained in~$\wt\Sigma$.

Let us fix a nonempty subset $I\subset J$, let $D_I^\circ$ the corresponding stratum of $D$, let $x_o\in D_I^\circ$ and let $D_I^\circ(x_o)$ the connected component of $D_I^\circ$ containing $x_o$. Let us fix a local holomorphic decomposition
\[
\bigl(\nb(x_o,X),\nb(x_o,D)\bigr)=(\Omega,D_\Omega)\times\nb(x_o,D_I^\circ),
\]
where~$\Omega$ is an open neighbourhood of $0$ in $\CC^\ell$ and $D_\Omega$ is the union of the coordinate hyperplanes in $\CC^\ell$. The category $\Mod_\hol\bigl((X,D_I^\circ(x_o)),D,\wt\Sigma\bigr)$ has been defined in Section \ref{sec:bonnestructureRH}, and we have the similar category $\Mod_\hol((\Omega,0),D_\Omega,\wt\Sigma_{x_o})$, where $\wt\Sigma_{x_o}$ is the restriction of $\wt\Sigma$ above $\partial\wt\Omega:=\varpi^{-1}(D_\Omega)$.

\begin{theoreme}\label{th:equivcat}
Set $G=\pi_1\bigl(D_I^\circ(x_o),x_o\bigr)$. There is a natural equivalence of categories:
\[
\Mod_\hol\bigl((X,D_I^\circ(x_o)),D,\wt\Sigma\bigr)\simeq G\textup{-}\Mod_\hol\bigl((\Omega,0),D_\Omega,\wt\Sigma_{x_o}\bigr).
\]
\end{theoreme}

\begin{proof}
We set $\partial\wt X_I^\circ(x_o):=\varpi^{-1}(D_I^\circ(x_o))$ and we denote similarly by $\wt\Sigma_I^\circ(x_o)$ the restriction of $\wt\Sigma$ above this set.

\begin{enumerate}
\item\label{enum:equivcat1}
By the Riemann-Hilbert correspondence (Theorem \ref{th:RHDIcirc}), we can replace the category on the left-hand side with that of Stokes-filtered local systems on $\partial\wt X_I^\circ(x_o)$ with associated $\ccI$-covering contained in $\wt\Sigma_I^\circ(x_o)$.

\item\label{enum:equivcat3}
Let $\pi:(E_I^\circ(x_o),y_o)\to (D_I^\circ(x_o),x_o)$ be a universal covering of $D_I^\circ(x_o)$ with base-point $y_o$ above $x_o$ and let $G=\Gal(\pi)$ be the corresponding Galois group. We consider the fibre-product diagram
\[
\xymatrix{
\partial\wt Y_I^\circ(x_o)\ar[d]\ar[r]&\partial\wt X_I^\circ(x_o)\ar[d]^\varpi\\
(E_I^\circ(x_o),y_o)\ar[r]^-\pi& (D_I^\circ(x_o),x_o)
}
\]
and we denote by $\pi^{-1}\wt\Sigma_I^\circ(x_o)$ the corresponding pullback $\pi^{-1}\ccI$-covering of $\partial\wt Y_I^\circ(x_o)$. Then the category considered in \eqref{enum:equivcat1} is equivalent to the category of $G$-Stokes-filtered local systems $(\cL_I^\circ,\cL_{I,\bbullet}^\circ)$ on $\partial\wt Y_I^\circ(x_o)$ with associated $\pi^{-1}\ccI$-covering contained in $\pi^{-1}\wt\Sigma_I^\circ(x_o)$. This is a standard argument.

\item\label{enum:equivcat4} (\Cf \cite[Th.\,4.13]{Mochizuki10b} and Remark \ref{rem:similar})
The sheaf-theoretic restriction functor is an equivalence from the latter category to the category of $G$-Stokes-filtered local systems $(\cL,\cL_\bbullet)$ on $(\partial\wt \Omega)_0\simeq(S^1)^\ell$ with associated $\ccI_{x_o}$-covering contained in~$\wt\Sigma_{x_o}$ (we identify here $(\pi^{-1}\ccI)_{y_o}$ with $\ccI_{x_o}$ and $\pi^{-1}\wt\Sigma_I^\circ(x_o)_{y_o}$ with $\wt\Sigma_{x_o}$). This proof will be reviewed in the appendix.
\item\label{enum:equivcat5}
By applying now the $G$-Riemann-Hilbert correspondence of Theorem \ref{th:RHDIcirc} in the reverse direction to $((\Omega,0),D_\Omega,\wt\Sigma_{x_o})$, one ends the proof of the theorem.\qedhere
\end{enumerate}
\end{proof}

\section{The irregularity complex}\label{sec:irr}
Our aim in this section is to show that, under the goodness assumption as above, the irregularity complex is determined by its restriction to the smooth part of $D$. More precisely, for every $I\subset J$, and for every connected component of $D_I^\circ$, we show that there exists a component $D_k$ of $D$ ($k\in I$) such that $\iota_I^{-1}\Irr_D\cM$ (on this connected component) is determined by $\iota_k^{-1}\Irr_D\cM$.

Let $(\cL,\cL_\bbullet)$ be the Stokes-filtered local system corresponding to a (germ of) good $D$\nobreakdash-meromorphic flat bundle $\cM$. We have $\cL=\wti{}^{-1}\bR\wtj_*\DR\cM_{|X\moins D}$, where
\[
\wti:\partial\wt X\hto\wt X\quad\text{and}\quad\wtj:X\moins D\hto\wt X
\]
are the natural closed and open inclusions. Let us denote by $\cA^\modD_{\wt X}$ (\resp $\cA^\rdD_{\wt X}$) the sheaf on $\wt X$ of holomorphic functions on $X\moins D$ having moderate growth (\resp rapid decay) along $\partial\wt X$. One can then define the moderate (\resp rapidly decaying) de~Rham complex $\DR^\modD\cM$ (\resp $\DR^\rdD\cM$) on $\partial\wt X$. With the goodness assumption, it is known that both have cohomology in degree zero at most. More precisely, the Riemann-Hilbert correspondence recalled in Section \ref{subsec:RHglobal} gives
\[
\cL_{\leq0}=\cH^0\DR^\modD\cM\quad\text{and}\quad\cH^j\DR^\modD\cM=0\text{ for }j\neq0. 
\]
We set $\cL^{>0}:=\cL/\cL_{\leq0}$, and similarly $\DR^{>\modD}\cM$ is defined as the cone of
\[
\DR^\modD\cM\to\wti{}^{-1}\bR\wtj_*\DR\cM_{|X\moins D},
\]
so that $\cL^{>0}=\cH^0\DR^{>\modD}\cM$ (and $\cH^k\DR^{>\modD}\cM=0$ for $k\neq0$).

\begin{proposition}\label{prop:varpiirr}
We have $\Irr_D\cM[1]=\bR\varpi_*\cL^{>0}$.
\end{proposition}

\begin{proof}
We have
\[
\bR\varpi_*\DR^\modD\cM=\DR\cM(*D)\quad\text{and}\quad\bR\varpi_*\bR\wtj_*\DR\cM_{|X\moins D}=\bR j_*\DR\cM_{|X\moins D},
\]
where $j:X\moins D\hto X$ is the inclusion. We then apply \cite[Def.\,3.4-1]{Mebkhout04}.
\end{proof}

\begin{remarque}[The irregularity complex $\Irr_D^*\cM$]\label{rem:Lneg}
Recall that Mebkhout also defined the irregularity complex $\Irr_D^*\cM$ in \cite{Mebkhout90} (\cf also \cite[Def.\,3.4-2]{Mebkhout04}), which is non-canonically isomorphic to the complex $\bR\cHom_{\cD_{X|D}}(\cM,\cQ_D)[-1]$, where $\cQ_D=\cO_{\wh{D}}/\cO_{X|D}$ (\cf\cite[Cor.\,3.4-4]{Mebkhout04}). Let us set $\cL_{\prec0}:=\cH^0\DR^\rdD\cM$. We then have
\begin{starequation}\label{eq:Lneg}
\bR\varpi_*\cL_{\prec0}\simeq\Irr_D^*\cM^\vee,
\end{starequation}%
where $\cM^\vee$ is the holonomic $\cD_X$-module dual to $\cM$. Indeed, According to \cite[(3.13)]{Kashiwara03} we have
\[
\DR(\cQ_D\overset{\bL}\otimes\cM)[-1]\simeq\bR\cHom_{\cD_{X|D}}(\cM^\vee,\cQ_D)[-1]\simeq\Irr^*_D\cM^\vee.
\]
On the other hand, as $\cQ_D$ is flat over $\cO_{X|D}$ (because $\cO_{\wh{X|D}}$ is faithfully flat over~$\cO_{X|D}$) and as $\bR\varpi_*\cA_{\wt X}^\rdD\simeq\cQ_D[-1]$, we have
\[
\DR(\cQ_D\otimes\cM)[-1]\simeq\DR(\cQ_D\overset{\bL}\otimes\cM)[-1]\simeq\bR\varpi_*\DR^\rdD\cM.
\]

We also notice that $\Irr_D^*\cM^\vee=\Irr_D^*\cM^\vee(*D)$ and $\cM^\vee(*D)$ is also a good $D$\nobreakdash-meromorphic flat bundle, which is identified with the dual $D$-meromorphic flat bundle $\cHom_{\cO_X(*D)}(\cM,\cO_X(*D))$.
\end{remarque}

Let us fix $I\subset J$. Near each point $x_o$ of $D_I^\circ$, there exists a local ramification $\rho:\nb(x_o)_{\bmd_I}\to\nb(x_o)$ along $D$ such that the pullback of $\cM$ has a good formal decomposition at each point in $\nb(x_o)_{\bmd_I}$. By the goodness assumption, there exists an index $k(x_o)\in I$ such that each nonzero $\varphi\in\Phi_{x_o}$ has a pole along $D_{k(x_o)}$: indeed, the set $\Phi_{x_o}\cup\{0\}$ is good, so in particular the pole divisors of each of its nonzero elements are totally ordered; the smallest such divisor is nonzero, and we can choose~$k(x_o)$ to be the index of a component of this divisor. One can choose this index constant along any connected component of $D_I^\circ$. For  simplicity, we denote by $k(I)$ the locally constant function $x_o\mto k(x_o)$ on~$D_I^\circ$. Similarly, we denote by $k'(I)$ any index $k'$ such that the following property is satisfied: any $\varphi\in\Phi_{x_o}$ having a pole along $D_{k'}$ has a pole along all the components of $D$ passing through $x_o$ (such a $k'$ exists, due to the goodness condition). Equivalently, the number of $\varphi\in \Phi_{x_o}$ having \emph{no pole} on $D_{k'}$ is maximum (this maximum could be zero).

For every subset $I\subset J$, we have a natural inclusion lifting~$\iota_I$:
\[
\wt\iota_I:\partial\wt X_I^\circ=\varpi^{-1}(D_I^\circ)\hto\varpi^{-1}(D)=\partial\wt X.
\]

\begin{proposition}\label{prop:irrrestrtilde}
Let us fix $I\subset J$ and let us set $k=k(I)$ for simplicity. Then the natural morphism $\wt\iota_I^{-1}\cL^{>0}\to\wt\iota_I^{-1}\bR\wt\iota_{k*}\,\wt\iota{}_k^{-1}\cL^{>0}$ is an isomorphism. The same property holds for $\cL_{\prec0}$ up to replacing $k(I)$ with $k'(I)$.
\end{proposition}

By applying $\bR\varpi_*$ and using Proposition \ref{prop:varpiirr}, we obtain:

\begin{corollaire}\label{cor:irrrestr}
With the notation as in Proposition \ref{prop:irrrestrtilde}, the natural morphism $\iota_I^{-1}\Irr_D(\cM)\to\iota_I^{-1}\bR\iota_{k*}\,\iota_k^{-1}\Irr_D(\cM)$ is an isomorphism. The same property holds for $\Irr_D^*(\cM)$ up to replacing $k(I)$ with $k'(I)$.\qed
\end{corollaire}

\begin{proof}[Proof of Proposition \ref{prop:irrrestrtilde}]
Since the morphism is globally defined, the proof that it is an isomorphism is a local question. We thus fix $x_o\in D_I^\circ$ and work in some neighbourhood $\nb(x_o)$ of $x_o$ that we may shrink if needed.

Let us first assume that $\cM=\cE^\varphi$ (\cf Section \ref{subsec:nota}) for some $\varphi\in\cO_{X,x_o}(*D)$.

-- If $\varphi=0$ in $\cO_{X,x_o}(*D)/\cO_{X,x_o}$, then $\cL^{>0}=0$ and there is nothing to prove.

-- If $\varphi\neq0$ in $\cO_{X,x_o}(*D)/\cO_{X,x_o}$, we set $\varphi(x)=u(x)/x^m$, where $u\in\cO_{X,x_o}$ satisfies $u(x_o)\neq0$, and $m_i\in\NN$ for $i\in I$. In particular, $m_{k(I)}\neq0$. We choose polar coordinates on $\varpi^{-1}(\nb(x_o))$ of the form $(\rho_1,\dots,\rho_\ell,\theta_1,\dots,\theta_\ell,(x_j)_{j\notin I})$ with $\rho_i\in[0,\epsilon)$. We can assume that, in these coordinates, $m_i\neq0$ for $i=1,\dots,p$, $m_i=0$ for $i=p+1,\dots,\ell$, and that $k(I)=\nobreak1$. Then, in these coordinates, $\varpi^{-1}(D\cap\nb(x_o))=\prod_{i=1}^\ell\rho_i=0$ and $\cL^{>0}$ is the constant sheaf of rank one on the closed subset of $\varpi^{-1}(D\cap\nb(x_o))$ defined by
\begin{equation}\label{eq:ux}
\begin{cases}
\sum_{i=1}^pm_i\theta_i\in\arg u(x)+[-\pi/2,\pi/2],\\
\prod_{i=1}^p\rho_i=0,
\end{cases}
\end{equation}
and it is zero outside this closed subset. Let us describe this closed subset. We set $x':=(x_j)_{j\notin I}\in\Delta_\epsilon^{n-\ell}$ (with $0<\epsilon\ll1$) and $(\rho,e^{\sfi\theta})\in[0,\epsilon)^\ell\times(S^1)^\ell$. We can write $u(x)=u(\rho,\theta,x')=u(x_o)e^{g(x)}$ with $g$ holomorphic and $g(0)=0$ and we set $e^{\sfi\theta_o}:=u(x_o)/|u(x_o)|$. A simple computation shows that, if $\epsilon>0$ is small enough, the map
\begin{align*}
[0,\epsilon)^\ell\times(S^1)^\ell\times\Delta_\epsilon^{n-\ell}&\To{(F,\rho,x')} S^1\times[0,\epsilon)^\ell\times\Delta_\epsilon^{n-\ell}\\
(\rho,\theta,x')&\mto\Big(\textstyle\prod_{i=1}^pe^{\sfi m_i\theta_i}\cdot e^{-\sfi(\theta_o+\im g(\rho,\theta,x'))},\rho,x'\Big)
\end{align*}
has everywhere maximal rank (in fact, we have $\partial F/\partial\theta_1(0,\theta,0)\neq0$ on $(S^1)^\ell$). By Ehresmann's theorem, the map $(F,\rho,x')$ is a $C^\infty$ fibration, which can be trivialized on contractible sets like $[-\pi/2,\pi/2]\times[0,\epsilon)^\ell\times\Delta_\epsilon^{n-\ell}$. For our topological computation, we can thus as well consider the situation where $u(x)$ is constant and replace $u(x)$ with $u(x_o)$ in \eqref{eq:ux}.

Each connected component of \eqref{eq:ux} is then homeomorphic to a product
\[
\partial[0,\epsilon)^p\times[a,b]\times (S^1)^{p-1}\times[0,\epsilon)^{\ell-p}\times(S^1)^{\ell-p}\times\Delta_\epsilon^{n-\ell}
\]
for suitable $a,b$. The trace of this set on $\varpi^{-1}(D_{k(I)}^\circ)$ is the set defined by $\prod_{j=2}^\ell\rho_j\neq0$. This is the subset
\begin{equation}\label{eq:ouv}
\{\rho_1=0\}\times(0,\epsilon)^{p-1}\times[a,b]\times (S^1)^{p-1}\times[0,\epsilon)^{\ell-p}\times(S^1)^{\ell-p}\times\Delta_\epsilon^{n-\ell}.
\end{equation}
Its closure is the subset
\begin{equation}\label{eq:ferm}
\{\rho_1=0\}\times[0,\epsilon)^{p-1}\times[a,b]\times (S^1)^{p-1}\times[0,\epsilon)^{\ell-p}\times(S^1)^{\ell-p}\times\Delta_\epsilon^{n-\ell}.
\end{equation}
The ordinary pushforward of the constant sheaf on~\eqref{eq:ouv} by the open inclusion $\eqref{eq:ouv}\hto\nobreak\eqref{eq:ferm}$ is the constant sheaf on \eqref{eq:ferm} and the higher pushforwards vanish. Since $\varpi^{-1}(D_I)$ is the subset of \eqref{eq:ferm} defined by $\rho_i=0$ for $i=2,\dots,\ell$, the restriction of the latter sheaf to $\varpi^{-1}(D_I)$ is the constant sheaf on $\varpi^{-1}(D_I)$, and the morphism $\wt\iota_I^{-1}\cL^{>0}\to\wt\iota_I^{-1}\bR\wt\iota_{k*}\,\wt\iota{}_k^{-1}\cL^{>0}$ is nothing but the identity $\CC_{\varpi^{-1}(D_I)}\to\CC_{\varpi^{-1}(D_I)}$, proving the proposition in this case.

Let us now consider the general case. As already said, the question is local, and we argue now locally on $\partial\wt X$. One can then reduce the question to the non-ramified case and apply the higher dimensional Hukuhara-Turrittin theorem (\cf\eg\cite[Th.\,12.5]{Bibi10}). Let $\cA_{\wt X}$ denote the sheaf of $C^\infty$ functions on $\wt X$ which are holomorphic on~$X^*$ in some neighbourhood of~$\wt x_o$. We can thus assume that $\cA_{\wt X}\otimes\nobreak\varpi^{-1}\cM$ decomposes as the direct sum of terms \hbox{$\cA_{\wt X}\otimes\varpi^{-1}(\cE^\varphi\otimes\cR_\varphi)$}. By~induction on the rank, we can also assume that $\cR_\varphi$ has rank one, and locally on $\varpi^{-1}(D_I^\circ)$ the corresponding local system is trivial, so we can finally assume that $\cM=\cE^\varphi$, a case which was treated above.

The case of $\cL_{\prec0}$ is treated similarly by reducing to the case where $\cM=\cE^\varphi$. Assume first that $\varphi$ has poles along \emph{all} components of $D$ passing through $x_o$ (i.e., $p=\ell$). If we regard all sheaves considered above as external products of constant sheaves of rank one with respect to the product decomposition in \eqref{eq:ouv} and \eqref{eq:ferm}, the case of $\cL_{\prec0}$ is obtained by replacing $[-\pi/2,\pi/2]$ with the complementary open interval in \eqref{eq:ux}, and the corresponding sheaf $\CC_{[a,b]}$ with the sheaf $\CC_{(a',b')}$ for suitable $a',b'$ (\ie the extension by zero of the constant sheaf on $(a',b')$). Then the same argument as above applies to this case.

If the assumption on $\varphi$ does not hold, we have $\iota_I^{-1}\cL_{\prec0}=0$ (since $e^\varphi$ is not of rapid decay \emph{all along} $D$) as well as $\iota_{k'}^{-1}\cL_{\prec0}=0$, so the statement is obvious in this case.
\end{proof}

\section{Proof of Theorem \ref{th:main}}
\subsubsection*{The case $\ell=1$}
We first assume that $I=\{i\}$. The transversal slice $\Omega$ has dimension one and $D_\Omega=\{0\}$. Let us first prove a statement in dimension one. Let $(\cL,\cL_\bbullet)$ be a Stokes-filtered local system on $S^1$ and let $(\gr\cL,(\gr\cL)_\bbullet)$ be the associated graded Stokes-filtered local system. We denote by $\cN$ \resp $\cN'$ the corresponding meromorphic flat bundles on $(\Omega,0)$.

It is well-known that $\cH^k\Irr_{D_\Omega}(\cN)$ and $\cH^k\Irr_{D_\Omega}(\cN')$ have the same rank for any $k$, and vanish except for $k=1$, and similarly for $\Irr_D^*\cN^\vee$ and $\Irr_D^*\cN^{\prime\vee}$. They correspond to $H^0(S^1,\cL^{>0})$ and $H^0(S^1,\gr\cL^{>0})$ on the one hand, $H^1(S^1,\cL_{\prec0})$ and $H^1(S^1,\gr\cL_{\prec0})$ on the other hand (this is of course a particular case of Proposition~\ref{prop:varpiirr} and Remark \ref{rem:Lneg}).

\begin{lemme}\label{lem:LgrL}
There exists an isomorphism between the vector spaces $H^1(S^1,\cL_{\prec0})$ and $H^1(S^1,\gr\cL_{\prec0})$ such that, for any automorphism $\lambda$ of $(\cL,\cL_\bbullet)$, the induced automorphism of $H^1(S^1,\cL_{\prec0})$ corresponds, via this isomorphism, to the automorphism induced by $\gr\lambda$ on $H^1(S^1,\gr\cL_{\prec0})$. The same assertion holds for $H^0(S^1,\cL^{>0})$ and $H^0(S^1,\gr\cL^{>0})$ respectively.
\end{lemme}

\begin{proof}
We start with $\cL_{\prec0}$. Let us cover $S^1$ with open intervals $(U_\alpha)_{\alpha=1,\dots,N}$ such that
\begin{itemize}
\item
every open interval which contains at most one Stokes direction for every pair of distinct exponential factors (\cf\eg Example~1.4 in \cite{Bibi10}),
\item
the intersection of two intervals of the covering is an interval not containing any Stokes direction,\enlargethispage{\baselineskip}%
\item
there are no triple intersections of intervals of the covering.
\end{itemize}
Then this covering is a Leray covering for $\cL_{\prec0}$ (\cf\eg the proof of Lemma\,3.12 in \loccit), and moreover the only nonzero term of the associated \v{C}ech complex is the term in degree one. It follows that
\[
H^1(S^1,\cL_{\prec0})=\bigoplus_{\alpha=1,\dots,N}H^0(U_\alpha\cap U_{\alpha+1},\cL_{\prec0}),
\]
if we set $U_{N+1}=U_1$.

Recall that, on each interval $U_\alpha$, the Stokes-filtered local system $(\cL,\cL_\bbullet)$ is graded, \ie the Stokes filtration splits (\cf\eg Lemma\,3.12 in \loccit). Let us choose a splitting on $U_\alpha\cap U_{\alpha+1}$. Then Theorem 3.5 (and its proof) in \loccit\ shows that any automorphism $\lambda$ is graded with respect to the chosen splitting on $U_\alpha\cap U_{\alpha+1}$. It~follows that the action of the automorphism on $H^0(U_\alpha\cap U_{\alpha+1},\cL_{\prec0})$ is the same as the action of the associated graded automorphism on $H^0(U_\alpha\cap U_{\alpha+1},(\gr\cL)_{\prec0})$, so we have found a model where both actions are equal.

For $\cL^{>0}$ we argue by duality. Recall that the dual local system $\cL^\vee$ is naturally endowed with a Stokes-filtration $\cL^\vee_\bbullet$ (so that $(\cL^\vee,\cL^\vee_\bbullet)$ RH-corresponds to the dual meromorphic flat bundle), that $\cL^{>0}\simeq\cHom_\CC(\cL^\vee_{\prec0},\CC)$ (this is similar to \cite[Lem.\,2.16]{Bibi10}), and this isomorphism is compatible with grading. In particular, it induces isomorphisms
\[
H^0(S^1,\cL^{>0})\simeq H^1(S^1,\cL^\vee_{\prec0})^\vee\quad\text{and}\quad H^0(S^1,\gr\cL^{>0})\simeq H^1(S^1,\gr\cL^\vee_{\prec0})^\vee,
\]
and by the first point applied to $(\cL^\vee,\cL^\vee_\bbullet)$ we obtain a distinguished isomorphism between $H^0(S^1,\cL^{>0})$ and $H^0(S^1,\gr\cL^{>0})$. Let $\lambda$ be an automorphism of $(\cL,\cL_\bbullet)$, and let $\lambda^\vee$ be its dual. Then the first point applied to $\lambda^\vee$ gives the desired property for $\lambda$.
\end{proof}

\begin{proof}[End of the proof of Theorem \ref{th:main} in the case $\ell=1$]
We set $I=\{i\}$, $G=\pi_1(D_i^\circ,x_o)$. By Lemma \ref{lem:LgrL}, given a Stokes-filtered local system $(\cL,\cL_\bbullet)$ endowed with a $G$\nobreakdash-action (\ie a representation $G\to\Aut(\cL,\cL_\bbullet)$), there exists an isomorphism between $H^0(S^1,\cL^{>0})$ and $H^0(S^1,\gr\cL^{>0})$, \resp $H^1(S^1,\cL_{\prec0})$ and $H^1(S^1,\gr\cL_{\prec0})$, so that the induced $G$\nobreakdash-action on $H^0(S^1,\cL^{>0})$ is transformed into the induced graded $G$\nobreakdash-action on $H^0(S^1,\gr\cL^{>0})$, and the induced $G$-action on $H^1(S^1,\cL_{\prec0})$ into the induced graded $G$\nobreakdash-action on $H^1(S^1,\gr\cL_{\prec0})$.

Recall now that $\Irr_D\cM$ is a complex whose cohomology is locally constant on each $D_I^\circ$. On $D_i^\circ$ it reduces to the local system $\cH^1\Irr_{D_i^\circ}\cM$. If we consider the $G$\nobreakdash-Stokes-filtered local system $(\cL,\cL_\bbullet)$ on $S^1$ corresponding to $\cM_{|D_i^\circ}$ by (the proof of) Theorem \ref{th:equivcat}, then $\cH^1\Irr_{D_i^\circ}\cM$ is the local system corresponding to $G$-vector space $H^0(S^1,\cL^{>0})$ that this $G$-Stokes-filtered local system defines. We argue similarly with~$\cM_i^\circ$ and $(\gr\cL,\gr\cL_\bbullet)$, so that the desired isomorphism follows from Lemma~\ref{lem:LgrL}, as explained above. The argument for $\Irr_{D_i^o}^*\cM$ is identical.
\end{proof}

\subsubsection*{The case $\ell\geq2$}
When $\ell=\#I\geq2$, the structure of a Stokes-filtered local system on $(S^1)^\ell$ is more difficult to analyze, although it shares many properties with the case $\ell=1$ (\cf\eg\cite[\S9.e]{Bibi10}). This is why we use another argument. Namely, Proposition \ref{prop:varpiirr} enables us to deduce the case $\ell\geq2$ from the case $\ell=1$.

We set $k=k(I)$ as defined after Proposition \ref{prop:varpiirr}. Let $\nb(D_I^\circ)$ be an open neighbourhood of $D_I^\circ$ in $X$ on which $\cM_I^\circ$ is defined. We claim that
\[
\iota_k^{-1}\cM_I^\circ=\cM_k^\circ{}_{|\nb(D_I^\circ)}.
\]
Indeed, this follows from the uniqueness of $\cM_k^\circ$, and from the fact that $\cM_I^\circ$ also decomposes after ramification along $D$ at each point of $\nb(D_I^\circ)\cap D_i^\circ$ if this neighbourhood is chosen small enough. We then have
\begin{align*}
\Irr_D(\iota_k^{-1}\cM_I^\circ)&\simeq\Irr_D(\cM_k^\circ)_{|\nb(D_I^\circ)}\\
&\simeq\Irr_D(\iota_k^{-1}\cM)_{|\nb(D_I^\circ)}\quad\text{(case $\ell=1$)},
\end{align*}
and therefore, by applying $\iota_I^{-1}\bR\iota_{k*}$,
\[
\iota_I^{-1}\bR\iota_{k*}\iota_k^{-1}\Irr_D(\cM_I^\circ)\simeq\iota_I^{-1}\bR\iota_{k*}\iota_k^{-1}\Irr_D(\cM).
\]
The assertion of Theorem \ref{th:main} for $\Irr_D$ now follows from Corollary \ref{cor:irrrestr}, applied both to $\cM$ and~$\cM_I^\circ$. The case of $\Irr_D^*$ is completely similar.\qed

\appendix
\refstepcounter{section}
\section*{Appendix. Some properties of Stokes-filtered local systems}
In this appendix we keep the setting of Section\,\ref{sec:irr}. We review in Proposition \ref{prop:similar} the proof of \cite[Th.\,4.13]{Mochizuki10b}: by choosing the projection to $D_I^\circ$ of a tubular neighbourhood of $D_I^\circ$ in $X$ and its fibre product over $D_I^\circ$ with a universal covering of~$D_I^\circ$, we are in the situation of \loccit\ except that we do not assume that the $C^\infty$ fibration is topologically trivial. Remark \ref{rem:similar} will then provide the main result used in Step \ref{enum:equivcat4} of the proof of Theorem \ref{th:equivcat}. We will also review some other essential results which are proved in \loccit

\subsection{Grading of a Stokes-filtered local system}\label{subsec:grading}
The result in this subsection is local with respect to $D$, hence we allow a ramification around the components of $D$. We fix a nonempty subset $I\subset J$. We fix a simply connected open set $U_I^\circ\subset D_I^\circ$.

We assume that $(\cL,\cL_\bbullet)$ is non-ramified in the neighbourhood of $U_I^\circ$. The covering~$\wt\Sigma_I^\circ$ can then be trivialized on $U_I^\circ\times(S^1)^\ell=\varpi^{-1}(U_I^\circ)$, and we set
\[
\wt\Sigma_I^\circ=\Phi\times U_I^\circ\times(S^1)^\ell,
\]
where $\Phi$ is a finite subset of $\Gamma\bigl(U_I^\circ,(\cO_X(*D)/\cO_X)_{|U_I^\circ}\bigr)$. Moreover, by the goodness assumption on $\wt\Sigma$, $\Phi$ is a good set, namely, for every pair $\varphi\neq\psi$, the divisor of $\varphi-\psi$ is negative. The set $\St(\varphi,\psi)\subset U_I^\circ\times(S^1)^\ell$ of Stokes directions is smooth over~$U_I^\circ$ with fibers equal to a union of translated codimension-one subtori
\begin{equation}\label{eq:grading}
\St(\varphi,\psi)_x=\Big\{(\theta_1,\dots,\theta_\ell)\in(S^1)^\ell\mid\ts\sum_j m_j\theta_j-\arg c(x)=\pm\pi/2\bmod2\pi\Big\},
\end{equation}
where $c(x)$ is an invertible holomorphic function on $U_I^\circ$ and $(m_1,\dots,m_\ell)\in\NN^\ell\moins\{0\}$. We denote by $\St(\Phi)$ the union of the subsets $\St(\varphi,\psi)$ for all pairs $\varphi\neq\psi\in\Phi$.

Let us fix
\[
\theta_o=(\theta_{o,1},\dots,\theta_{o,\ell})\in(S^1)^\ell\quad\text{and}\quad \alpha_1,\dots,\alpha_\ell\in\NN^*
\]
such that $\gcd(\alpha_1,\dots,\alpha_\ell)=1$. The map $\theta\mto(\alpha_1\theta+\theta_{o,1},\dots,\alpha_\ell\theta+\theta_{o,\ell})$ embeds~$S^1$ in~$(S^1)^\ell$. In the following, $S^1_{\alphag,\theta_o}$ denotes this circle.

\begin{proposition}\label{prop:grading}
Let $A^\circ$ be an open interval of length $<2\pi$ in $S^1_{\alphag,\theta_o}$ and let $A$ be its closure. Assume that $A$ satisfies the following property.
\begin{itemize}
\item
For every $x\in U_I^\circ$ and every pair $\varphi\neq\psi\in\Phi$,
\[
\#\bigl(A\cap\St(\varphi,\psi)\bigr)=\#\bigl(A^\circ\cap\St(\varphi,\psi)\bigr)\leq1.
\]
If moreover $U_I^\circ$ is contractible, then $(\cL,\cL_\bbullet)$ is graded when restricted to a sufficiently small neighbourhood $U_I^\circ\times\nb(A)$ in $U_I^\circ\times(S^1)^\ell$.
\end{itemize}
\end{proposition}

\begin{proof}
We first prove that, for every $\varphi\in\Phi$, we have $H^k(U_I^\circ\times A,\cL_{<\varphi})=0$ for $k\geq1$. Note that, since $\varpi:U_I^\circ\times A\to U_I^\circ$ is proper, $R^k\varpi_*\cL_{<\varphi|U_I^\circ\times A}$ is compatible with base change, hence its germ at $x$ is equal to $H^k(A,\cL_{<\varphi|\{x\}\times A})$. By our assumption on~$A$, this is also equal to $H^k(A^\circ,\cL_{<\varphi|\{x\}\times A^\circ})$, and by the proof of \cite[Lem.\,9.26]{Bibi10}, this is zero for $k\geq1$. As a consequence, $R^k\varpi_*\cL_{<\varphi|U_I^\circ\times A}=0$ for $k\neq0$.

We argue as in \loccit\ to obtain that $(\cL,\cL_\bbullet)$ is graded in the neighbourhood of $\{x\}\times A$ for every $x\in U_I^\circ$. In particular, it is easy to check that $\varpi_*\cL_{<\varphi|U_I^\circ\times A}$ is locally constant, hence constant, on $U_I^\circ$. Since $U_I^\circ$ is assumed contractible, we obtain the vanishing of $H^k(U_I^\circ\times A,\cL_{<\varphi})$ ($k\geq1$). Using once more the argument of \loccit, we obtain the grading property all over $U_I^\circ\times A$, hence in some open neighbourhood of it.
\end{proof}

By mimicking the proof of \cite[Th.\,3.5\,\&\,Prop.\,9.21]{Bibi10}, we also obtain the following proposition.

\begin{proposition}\label{prop:gradedmorphism}
Let $\lambda:(\cL,\cL_\bbullet)\to(\cL',\cL'_\bbullet)$ between Stokes-filtered local systems as considered in the beginning of this subsection with the same set $\Phi$. For $A$ as in Proposition \ref{prop:grading}, there exist gradings of both Stokes-filtered local systems on $U_I^o\times\nb(A)$ with respect to which $\lambda$ is graded.\qed
\end{proposition}

\subsection{Closedness}\label{subsec:closedness}
Let $U_I^\circ$ be an open subset of $D_I^\circ$ with closure $\ov{U_I^\circ}$ in $D_I^\circ$ and boundary $\partial U_I^\circ$, and let $j:U_I^\circ\hto\ov{U_I^\circ}$ and $\wtj:\varpi^{-1}(U_I^\circ)\to\varpi^{-1}(\ov{U_I^\circ})$ be the open inclusions. Let $(\cL,\cL_\bbullet)$ be a Stokes-filtered local system on $\varpi^{-1}(U_I^\circ)$ with associated covering contained in $\wt\Sigma_I^\circ{}_{|U_I^\circ}$. Assume that
\begin{enumerate}
\item[$(*)$]
any point $x\in\partial U_I^\circ$ has a fundamental system of open neighbourhoods $V$ in $D_I^\circ$ such that $V\cap U_I^\circ$ and $V\cap\ov{U_I^\circ}$ are contractible.
\end{enumerate}

\begin{proposition}\label{prop:closedness}
Under this assumption, the functor $\wtj_*$ induces an equivalence between the category of Stokes-filtered local systems $(\cL,\cL_\bbullet)$ on $\varpi^{-1}(U_I^\circ)$ with associated $\ccI$\nobreakdash-covering contained in $\wt\Sigma_I^\circ{}_{|U_I^\circ}$, and the category of Stokes-filtered local systems on $\varpi^{-1}(\ov{U_I^\circ})$ with associated $\ccI$\nobreakdash-covering contained in $\wt\Sigma_I^\circ{}_{|\ov{U_I^\circ}}$, a quasi-inverse functor being the restriction $\wtj{}^{-1}$.
\end{proposition}

\begin{proof}
Since the functor is globally defined, the question is local near a point $x_o\in\nobreak\partial U_I^\circ$. Moreover, as in Section\,\ref{subsec:grading}, we can assume that $\wt\Sigma_I^\circ$ is a trivial covering on some neighbourhood of $x_o$. It is enough to prove the statement in the non-ramified case since, by~uniqueness the construction, it will descend by means of the Galois action of the ramification. We will work with the corresponding set $\Phi$ of exponential factors.

Firstly, we note that Assumption~$(*)$ also holds for $\varpi^{-1}(\ov{U_I^\circ})$, since any point in $\varpi^{-1}(x)$ has a fundamental systems of neighbourhoods of the form of the product of neighbourhoods~$V$ with a product of $\ell$ open intervals. It follows that the local system~$\cL$ extends in a unique way as a local system on $\varpi^{-1}(\ov{U_I^\circ})$, and the latter is~$\wtj_*\cL$. Similarly, a morphism between local systems extends in a unique way by the functor $\wtj_*$. The same property holds for the local systems $\gr_\psi\cL$ for $\psi\in\Phi$.

Let us first show that the functor $\wtj_*$ takes values in the category of Stokes-filtered local systems. For a pair $\varphi\neq\psi\in\Phi$, we denote by $\beta_{\psi\leq\varphi}$ the functor composed of the restriction to the open subset where $\varphi\leq\psi$ (\ie $\reel(\varphi-\psi)<0$) and the extension by zero to the whole space. The point is to check that every $\wtj_*\cL_{\leq\varphi}$ decomposes as $\bigoplus_{\psi\in\Phi}\beta_{\psi\leq\varphi}\wtj_*\gr_\psi\cL$ in the neighbourhood of every point $(x_o,\theta_o)$ of $\varpi^{-1}(x_o)$. If we fix a small interval $A^\circ$ containing this point as in Proposition \ref{prop:grading}, we find that, according to this proposition and Assumption~$(*)$,
\begin{equation}\label{eq:decomp}
\cL_{\leq\varphi|(V\cap U_I^\circ)\times\nb(A^\circ)}\simeq\bigoplus_{\psi\in\Phi}\beta_{\psi\leq\varphi}(\gr_\psi\cL)_{|(V\cap U_I^\circ)\times\nb(A^\circ)}.
\end{equation}
We are thus reduced to checking that, for a local system $L$, the natural morphism $\beta_{\psi\leq\varphi}L\to\wtj_*\beta_{\psi\leq\varphi}\wtj{}^{-1}L$ is an isomorphism: we will apply this to the local system $L=\wtj_*(\gr_\psi\cL)_{|(V\cap U_I^\circ)\times\nb(A^\circ)}$ for any $\psi$. The question is then local, and we can work in the neighbourhood of $(x_o,\theta_o)$, with the constant sheaf of rank one as the given local system.

If $(x_o,\theta_o)\notin\St(\varphi,\psi)_{x_o}$, the result is easy. We will thus focus on the case where $(x_o,\theta_o)\in\St(\varphi,\psi)_{x_o}$. This can be written as $\sum m_j\theta_{o,j}-\arg c(x_o)=\pm\pi/2$. We will consider the case $+\pi/2$, the other one being similar. We need to check that the germ at $(x_o,\theta_o)$ of $\wtj_*\wtj{}^{-1}\beta_{\psi\leq\varphi}\CC$ is zero for any such $(x_o,\theta_o)$. For that purpose, it is enough to prove that, for small enough closed neighbourhoods $V$ of $x_o$ and $\nb(\theta_o)$ of $\theta_o$, the cohomology of the sheaf on
\begin{equation}\label{eq:closedness}
(V\times\nb(\theta_o))\cap\bigl\{\ts\sum m_j\theta_j-\arg c(x)\in[\pi/2-\epsilon,\pi/2]\bigr\}
\end{equation}
which is zero on
\[
(V\times\nb(\theta_o))\cap\bigl\{\ts\sum m_j\theta_j-\arg c(x)=\pi/2\bigr\}
\]
and constant on the complementary set, is zero for $0<\epsilon\ll1$ and $V$ small enough. We~can regard $\sum m_j\theta_j-\arg c(x_o)-\pi/2$ as a coordinate $\theta'$ near $\theta_o$ vanishing at~$\theta_o$, and we can choose the neighbourhood $\nb(\theta_o)$ of the form $[-2\epsilon,2\epsilon]\times[-2\epsilon,2\epsilon]^{\ell-1}$ accordingly. For $V$ small enough, the set \eqref{eq:closedness} is a topological fibration above $V$, and the fiber over $x\in V$ is the product of $[-2\epsilon,2\epsilon]^{\ell-1}$ with the interval
\[
\theta'\in\arg c(x)-\arg c(x_o)+[-\epsilon,0].
\]
Since the projection to $V$ is proper, the base change formula shows that the pushforward to $V$ of this sheaf is identically zero, as the cohomology with compact support of a semi-closed interval is zero. Hence its global cohomology~on \eqref{eq:closedness} is also zero.

The next step is to show that the extension by $\wtj_*$ of a morphism $\lambda$ between Stokes-filtered local systems is compatible with the Stokes filtration. The question is local, and we can assume that the morphism $\lambda$ is graded on $(V\cap U_I^\circ)\times\nb(A^\circ)$, according to Proposition \ref{prop:gradedmorphism}. Then $\wtj_*\lambda$ is also graded on this open set with respect to the Stokes filtration constructed above, and is thus also Stokes-filtered.

Once the functor $\wtj_*$ is defined, that it is essentially surjective is proven similarly, since in the neighbourhood of any point $(x_o,\theta_o)$ the sheaves $\cL_{\leq\varphi}$ are given by a formula like \eqref{eq:decomp}.

The full faithfulness follows from the full faithfulness for the underlying local systems.
\end{proof}

\subsection{Openness}\label{subsec:openness}
We keep the notation as above.

\begin{proposition}\label{prop:openness}
Let $x_o\in D_I^\circ$ and let $(\cL,\cL_\bbullet)_{x_o}$ be a Stokes-filtered local system on $\varpi^{-1}(x_o)\simeq(S^1)^\ell$ with associated $\ccI$-covering contained in $\wt\Sigma_{I,x_o}^\circ$. Then there exists an open neighbourhood $\nb(x_o)$ in $D_I^\circ$ such that $(\cL,\cL_\bbullet)_{x_o}$ extends in a unique way as a Stokes-filtered local system on $\varpi^{-1}(\nb(x_o))\simeq\nb(x_o)\times(S^1)^\ell$ with associated $\ccI$-covering contained in $\wt\Sigma_{I}^\circ{}_{|\nb(x_o)}$. Any morphism $(\cL,\cL_\bbullet)_{x_o}\to(\cL',\cL'_\bbullet)_{x_o}$ between such objects also extends locally in a unique way.
\end{proposition}

\begin{proof}
The problem is local on $D_I^\circ$ and, by the uniqueness of the extension of morphisms, one can reduce the proof to the non-ramified case. We can therefore assume that $\Sigma_I^\circ=\Phi\times \nb(x_o)$. Moreover, the unique extension of local systems and morphisms between them is clear, so the question reduces to checking that Stokes filtrations extend as well, and that the extended morphism between the extended local systems is compatible with the extended Stokes filtrations.

By Proposition \ref{prop:grading}, we can cover $(S^1)^\ell=\varpi^{-1}(x_o)$ by simply connected open sets~$U_\alpha$ such that, for every~$\alpha$, there exists a neighbourhood $V_\alpha$ of the compact subset~$\ov U_\alpha$ and an isomorphism
\begin{equation}\label{eq:Stokesdec}
\cL_{x_o|V_\alpha}\simeq \bigoplus_{\varphi\in\Phi}\gr_\varphi\cL_{x_o|V_\alpha},
\end{equation}
and the Stokes filtration on $V_\alpha$ is given by
\begin{equation}\label{eq:Stokesfilgr}
\cL_{x_o,\leq\varphi|V_\alpha}\simeq\bigoplus_{\psi\in\Phi}\beta_{\psi\leq\varphi}\gr_\psi\cL_{x_o|V_\alpha}.
\end{equation}
The transition maps $\lambda_{\alpha\beta}$ for \eqref{eq:Stokesdec} on $V_{\alpha\beta}:=V_\alpha\cap V_\beta$ satisfy the cocycle condition and are compatible with the Stokes filtration, that is, $\lambda_{\alpha\beta}^{\psi,\varphi}:\gr_\psi\cL_{x_o|V_{\alpha\beta}}\to\gr_\varphi\cL_{x_o|V_{\alpha\beta}}$ is zero unless $\psi\leq\varphi$ on $V_{\alpha\beta}$.

Let us shrink $\nb(x_o)$ to a contractible open neighbourhood such that, for all \hbox{$\psi\neq\varphi\in\Phi$}, $\psi<\varphi$ on $V_{\alpha\beta}$ implies $\psi<\varphi$ on $\nb(x_o)\times U_{\alpha\beta}$. The local system $\gr_\varphi\cL_{x_o|U_\alpha}$ extends in a unique way to a local system $\gr_\varphi\cL_{|\nb(x_o)\times U_\alpha}$ on $\nb(x_o)\times U_\alpha$, and so do the morphisms $\lambda_{\alpha\beta}^{\psi\varphi}$, which satisfy thus the cocycle condition. In particular, if such an extension $\lambda_{\alpha\beta}^{\psi\varphi}$ is non-zero at one point of $\nb(x_o)\times U_{\alpha\beta}$, it is nonzero everywhere on this open set and we have $\psi<\varphi$ on this open set. Let us set $\cL_{|\nb(x_o)\times U_\alpha}:=\bigoplus_{\varphi\in\Phi}\gr_\varphi\cL_{|\nb(x_o)\times U_\alpha}$, that we equip with the Stokes filtration given by a formula similar to \eqref{eq:Stokesfilgr}. It follows that $\lambda_{\alpha\beta}$ is compatible with the Stokes filtrations. We regard now $\lambda_{\alpha\beta}$ as gluing data. The cocycle condition shows that they define a local system $\cL$ on $\varpi^{-1}(\nb(x_o))$ whose restriction to $\varpi^{-1}(x_o)$ is isomorphic to~$\cL$. It is thus uniquely isomorphic to the unique extension of $\cL_{x_o}$. Moreover, due to the compatibility with the Stokes filtrations, the latter also glue correspondingly as a Stokes filtration $\cL_\bbullet$ of this local system, and its restriction to $\varpi^{-1}(x_o)$ is equal to $\cL_{x_o \bbullet}$.

Let $\mu_{x_o}:(\cL,\cL_\bbullet)_{x_o}\to(\cL',\cL'_\bbullet)_{x_o}$ be a morphism. We can choose the covering $(U_\alpha)$ and the decomposition \eqref{eq:Stokesdec} so that each $\mu_{x_o,\alpha}$ is graded (\cf\cite[Prop.\,9.21]{Bibi10}). It extends uniquely as a morphism $\mu:\cL_{|\nb(x_o)\times U_\alpha}\to\cL'_{|\nb(x_o)\times U_\alpha}$, and it is graded with respect to the corresponding decompositions \eqref{eq:Stokesdec}. It follows that $\mu$ is strictly compatible with the Stokes filtrations $\cL_\bbullet$ and $\cL'_\bbullet$, where these Stokes-filtered local systems $(\cL,\cL_\bbullet)$ and $(\cL',\cL'_\bbullet)$ are obtained as in the first part.

We can now prove the uniqueness (\ie up to unique isomorphism) of $(\cL,\cL_\bbullet)$ constructed in the first part: the identity automorphism $(\cL,\cL_\bbullet)_{x_o}$ extends in a unique way as an isomorphism between two such extensions.
\end{proof}

\subsection{An equivalence of categories}
We will use the notation as in Section \ref{sec:equivcat}. Let $\pi:(E_I^\circ(x_o),y_o)\to (D_I^\circ(x_o),x_o)$ be a universal covering of $D_I^\circ(x_o)$ with base point $y_o$ above $x_o$, and let $\partial\wt Y_I^\circ(x_o)$ be the pullback of $\partial\wt X_I^\circ(x_o)$ by~$\pi$.

\begin{proposition}\label{prop:similar}
The restriction functor
\begin{itemize}
\item
from the category of Stokes-filtered local systems on $\partial\wt Y_I^\circ(x_o)$ with associated $\pi^{-1}\ccI$-covering contained in $\pi^{-1}\wt\Sigma_I^\circ(x_o)$
\item
to the category of Stokes-filtered local systems on $(\partial\wt \Omega)_0\simeq(S^1)^\ell$ with associated $\ccI_{x_o}$-covering contained in $\wt\Sigma_{x_o}$
\end{itemize}
is an equivalence.
\end{proposition}

\begin{proof}
Let $\Gamma:[0,1]^2\to E_I^\circ(x_o)$ be a continuous map sending $(0,0)$ to $y_o$. We pullback by $\Gamma$ the data from the first item of the proposition. Let us consider the subset of~$[0,1]$ consisting of $\epsilon$'s such that the equivalence of the proposition holds with respect to the restriction corresponding to the inclusion $(0,0)\in[0,\epsilon]^2$. Propositions \ref{prop:closedness} and \ref{prop:openness} imply that this set is open and closed, and contains $0$, hence it is equal to $[0,1]$. This shows that one can uniquely extend an object in the second category to an object in the first category along paths starting from $y_o$ and that this extension does not depend on the choice of the path. A similar assertion holds for morphisms.
\end{proof}

\begin{remarque}\label{rem:similar}
The uniqueness of the extension of morphisms enables one to obtain the equivalence between the corresponding $G$-equivariant categories, and this gives the implication $\eqref{enum:equivcat3}\implique\eqref{enum:equivcat4}$ in the proof of Theorem \ref{th:equivcat}.
\end{remarque}

\backmatter
\providecommand{\eprint}[1]{\href{http://arxiv.org/abs/#1}{\texttt{arXiv\string:\allowbreak#1}}}
\providecommand{\hal}[1]{\href{https://hal.archives-ouvertes.fr/hal-#1}{\texttt{hal-#1}}}
\providecommand{\doi}[1]{\href{http://dx.doi.org/#1}{\texttt{doi\string:\allowbreak#1}}}
\providecommand{\og}{``}
\providecommand{\fg}{''}
\providecommand{\smfphdthesisname}{Th\`ese}

\end{document}